\documentclass{amsart}
\usepackage{amsmath,amssymb,amsthm,amscd}

\def\p{\partial}
\def\R{\mathbb{R}}
\def\C{\mathbb{C}}

\def\i{\sqrt{-1}}

\def\t{\triangle}

\def\G{{\mathfrak G}}
\def\g{{\mathfrak g}}

\def\h{{\mathfrak h}}

\def\cD{\mathcal D}

\def\cF{{\mathcal F}}
\def\cG{{\mathcal G}}

\def\cH{{\mathcal H}}

\def\cJ{{\mathcal J}}
\def\cK{{\mathcal K}}

\def\cM{{\mathcal M}}

\def\cP{{\mathcal P}}
\def\cQ{{\mathcal Q}}

\def\cS{{\mathcal S}}

\def\cY{{\mathcal Y}}

\numberwithin{equation}{section}

\newtheorem{prop}{Proposition}[section]
\newtheorem{theo}[prop]{Theorem}

\newtheorem{rmk}[prop]{Remark}

\newtheorem{defi}[prop]{Definition}

\newtheorem{pr}[prop]{Problem}

\begin{document}
\title[Scalar curvature on Sasaki manifolds]{On the transverse Scalar Curvature of a Compact Sasaki Manifold}\author{Weiyong HE}

\begin{abstract} We show that the {\it standard picture} regarding the notion of stability of constant scalar curvature metrics in K\"ahler geometry described by S.K. Donaldson \cite{Donaldson97, Donaldson99}, which involves the geometry of infinite-dimensional groups and spaces,  can be applied to the constant scalar curvature  metrics in Sasaki geometry with only few modification. We prove that the space of Sasaki metrics is an infinite dimensional symmetric space and that the transverse scalar curvature of a Sasaki metric is a moment map of the strict contactomophism group.  \end{abstract}

\address{Department of Mathematics, University of Oregon, Eugene, OR, 97403}
\email{whe@uoregon.edu}

\thanks{The author is partially supported by National Science Foundation}


\date{}

\dedicatory{}


\maketitle

\section{Introduction}
Sasaki geometry, in particular Sasaki-Einstein manifolds, has been extensively studied.  Readers are referred to recent monograph Boyer-Galicki \cite{BG}, and recent survey paper Sparks \cite{Sparks} and the references in for history, background and recent progress of Sasaki geometry and Sasaki-Einstein manifolds. 

Sasaki geometry is often described as an odd-dimensional analogue of K\"ahler geometry. Some important results in K\"ahler geometry have now a counterpart in the Sasaki context.  For example, Calabi's extremal problem in K¬ahler geometry has been extended to the Sasaki manifolds and some important results have been obtained \cite{BGS1, BGS2, FOW}.  In this note we shall discuss the notion of stability of constant scalar curvature metrics in Sasaki context. One can see that the picture regarding the stability of constant scalar curvature described by S.K. Donaldson \cite{Donaldson97, Donaldson99} can be carried over to the Sasaki setting with only slight modification. 

It is well known in K\"ahler geometry that there are  obstructions to the existence of extremal metrics in general \cite{M, L, Futaki83, Calabi85};  there is now a very beautiful picture regarding the existence of constant scalar curvature metrics on algebraic varieties, which is conjectured to be equivalent to certain  notion of  stability in algebraic geometry \cite{Yau90, Tian97, Donaldson02}; this picture can also be generalized to extremal metrics \cite{Mabuchi04, Sze}.  There has also been important progress to understand obstructions of canonical metrics in Sasaki geometry \cite{ GMSY, FOW, BGS1}. In particular, Gauntlett-Martelli-Sparks-Yau \cite{GMSY} has proved the {\it Lichnerowicz obstruction} and the {\it Bishop obstruction} to existence of Sasaki Einstein metrics.  Their results  provide new obstructions to  existence of K\"ahler-Einstein metrics on K\"ahler orbifolds. Recently Ross-Thomas \cite{RT} studied  stability of K\"ahler orbifolds and constant scalar curvature metrics on K\"ahler orbifolds.

In \cite{Donaldson97, Donaldson99},  Donaldson suggested that the stability of K\"ahler-Einstein metrics \cite{Yau90, Tian97} fits into a general framework, involving the geometry of infinite dimensional groups and spaces; from this point of view,  the notion of stability  can be extended to, for example,  the metrics of constant scalar curvature. 
One of the key observations relating  canonical metrics in K\"ahler geometry  with the geometry of infinite dimensional groups and spaces is that the scalar curvature of a K\"ahler manifold $(V, \omega)$ can be interpreted as a moment map of symplectomorphism group of $(V, \omega)$ \cite{Fujiki90, Donaldson97}. We shall show similar pictures described  in \cite{Donaldson97, Donaldson99} can be extended to Sasaki geometry with only slight modification.  Our discussion follows closely the discussion by S.K. Donaldson \cite{Donaldson97, Donaldson99} in (almost) K\"ahler setting. While the picture described by Donaldson stimulates tremendous work in Kahler geometry; it might  still be worth to show that the similar picture holds for the Sasaki case, even though there is no new idea beyond \cite{Donaldson97, Donaldson99} and there is only few modification to prove the results.

In Section 3 we consider the space of Sasaki metrics studied in \cite{GZ1, GZ2} , which is the analogue of the space of K\"ahler metrics as in \cite{Mabuchi87, Semmes96, Donaldson99}.  Guan-Zhang \cite{GZ1, GZ2} studied the geometry of the space of Sasaki metrics, in particular the geodesic equations. They proved, among others,  the regularity results of the Dirichlet problem of  the geodesic equation with applications in Sasaki geometry, parallel  the results of Chen \cite{Chen00} and Calabi-Chen \cite{CalabiChen} in the space of K\"ahler metrics. As a consequence of their results, the space of Sasaki metrics is a non-postively curved metric space. We observe that  the space of Sasaki metrics is an infinite dimensional symmetric space as in the K\"ahler setting  \cite{Mabuchi87, Semmes96, Donaldson99}.  Moreover, $\cH$  can be viewed as the dual space of the strict contactomorphism group $\cG$.

In Section 4 we show that the transverse scalar curvature of a Sasaki metric (K-contact metric) can  be interpreted as a moment map of strict contactomorphism group; hence the ``standard picture" described in \cite{Donaldson97, Donaldson99} can be extended to Sasaki geometry. Thus  it is not a surprise that the existence of canonical metrics in Sasaki geometry should also be relevant to the notion of stability, see \cite{RT} for recent progress. Following \cite{Donaldson99},  we shall present a problem regarding nonexistence of  constant transverse scalar curvature metrics, using the notion of  geodesic rays and the $K$-energy in Sasaki geometry \cite{Mabuchi86, GZ2}.

\section{Sasaki Geometry, the Contact Structure and the Transverse K\"ahler Structure}
A Sasaki manifold involves  many interesting structures, including its underlying contact structure and the transverse K\"ahler structure. In this section we recall these relevant structures on Sasaki manifolds, or more generally, K-contact manifolds.

Let $(M, \eta)$ be a compact contact manifold of dimension $2n+1$ with a contact 1-form $\eta$ such that $\eta\wedge (d\eta)^n$ is nowhere vanishing. The Reeb vector field $\xi$ is define uniquely by
\[
\eta(\xi)=1, \iota_\xi d\eta=0. 
\]
The $1$-dimensional foliation generated by $\xi$ is called the {\it Reeb foliation}. 
The contact form $\eta$ defines a vector sub-bundle $D$ of the tangent bundle $TM$ such that
 $\cD=ker(\eta)$  and $TM=\cD\oplus L\xi$, where $L\xi$ is the trivial line  bundle generalized by $\xi$.
A contactomorphism $f: M\rightarrow M$ is a diffeomorphism which satisfies $f^{*}\eta=\exp(F)\eta$ for some function $F: M\rightarrow \R$. The group $\G$ of contactomorphisms leaves  $\cD=ker(\eta)$ invariant. With the 1-form $\eta$ fixed we can also define the group $\cG$ of strict contactomorphisms by the condition $f^{*}\eta=\eta$; it is clear that $\cG$ is a subgroup of $\G$.  We denote the Lie algebra of 
$\G$ and $\cG$ by $\h$ and $\g$ respectively.  The Lie algebra  can be characterized by
\[
\begin{split}
\h&=\{X\in \Gamma(TM), L_X\eta=F\eta~ \mbox{for some}~ F\in C^\infty(M)\}, \\
\g&=\{X\in \Gamma(TM), L_X\eta=0\}. 
\end{split}
\]
A vector field $X$ is called contact if $L_X\eta= F \eta$ and it is called strict contact if $L_X\eta=0$. There is a Lie algebra isomorphism between the 
 space $C^\infty (M)$ of functions on $M$ and the Lie algebra $\h$,  see \cite{MS, BG} for the details. 
 For every function $H: M\rightarrow \R$, there exists a unique contact vector field $X=X_H: M\rightarrow TM$ which satisfies 
\[
H=\eta(X), \iota_X d\eta=dH(\xi)\eta-dH. 
\]
  The Possion bracket is then defined by
 \[
 \{F, H\}=\eta([X_F, X_H]).
 \]
 There is also a natural  $L^2$ inner product on $C^\infty(M)$:
\[
\langle f, h \rangle=\int_M f h d\mu,
\]
where $d\mu=(2^nn!)^{-1}\eta\wedge (d\eta)^n$ is a volume form  determined by $\eta$. However this $L^2$ product is not $\G$-invariant. Instead we consider the subgroup $\cG$ and the space $C^\infty_B(M)$ of basic functions on $M$, where 
\[
C_B^\infty(M)=\{\phi\in C^\infty(M): d\phi (\xi)=0\}. 
\]
Then $C^\infty_B(M)$ is a sub-algebra of $C^\infty(M)$ and it is isomorphic to the Lie algebra of $\g$ (for example, see \cite{BG} Section 6.1).  The $L^2$ inner product, restricted on $C^\infty_B(M)$, is $\cG$-invariant.  We can write $C^\infty_B(M)=C^\infty_{B, 0}(M)\oplus \R,$ where $C^\infty_{B, 0}(M)$ is the space of functions of integral zero, the $L^2$ orthogonal complement of the constants.  
Now the groups $\cG$ has a bi-invariant metric defined by $L^2$ inner product on its Lie algebra, so it furnishes an example of infinite dimensional symmetric space.  When $M$ is a Sasaki manifold, we shall show in next section that this space has a {\it negatively curved dual}. 

First we recall the {\it $K$-contact structure}. Let $(M^{2n+1}, \eta, \xi)$ be a compact contact manifold. 

\begin{defi} A $(1, 1)$-tensor field $\Phi: TM\rightarrow TM$ is called an  almost contact-complex structure if
\[
\Phi \xi=0, \Phi^2=-id+\xi\otimes \eta. 
\]
It is called an  K-contact-complex structure if in addition, $L_\xi \Phi=0$. An almost contact-complex structure is compatible with $\eta$ if
\[
d\eta(\Phi X, \Phi Y)=d\eta(X, Y); d\eta(X, \Phi X)>0, X\in \cD, X\ne 0. 
\]
\end{defi}
If $\Phi$ is compatible with $\eta$,  $(M, \eta, \Phi)$ defines a Riemannian metric 
\[
g(X, Y)=\frac{1}{2}d\eta(X, \Phi Y)+\eta(X)\eta(Y),
\]
and $(M, \eta, \xi, \Phi, g)$ is called a contact metric structure. This metric structure is called a K-contact metric structure if  $L_\xi \Phi=0$, which corresponds  to an almost K\"ahler metric on a symplectic manifold. A Sasaki structure can then be defined as a K-contact metric structure $(M, \eta, \xi, \Phi, g)$ where $\Phi$ satisfies an integrable condition
\[
\nabla \Phi(X, Y)=g(\xi, Y)X-g(X, Y)\xi. 
\] 
This definition is equivalent to the following characterization: a Riemannian  metric $(M, \eta, \xi, \Phi,  g)$ is called Sasaki if the cone metric $(C(M)=M\times \R_{+}, dr^2+r^2 g)$ is K\"ahler. The complex structure $J$ on $C(M)$ can be defined as
\[
J\left(r\frac{\p}{\p r}\right)=\xi, JY=\Phi Y-\eta(Y) r\frac{\p }{\p r}, ~Y\in \Gamma(TM). 
\]
For a Sasaki structure, there are two relevant K\"ahler structures. One is the K\"ahler cone structure on $C(M)$ and the other is  the transverse K\"ahler structure for the Reeb foliation. For our purpose, we shall also consider the {\it transverse almost K\"ahler} structure for a K-contact metric structure $(M, \eta, \xi, \Phi, g)$. The discussion is similar as in  the transverse K\"ahler structure \cite{BG, FOW}.
We shall describe the transverse almost K\"ahler structure both globally and locally.

Let $\{U_\alpha\}_{\alpha\in I}\subset \R^{2n+1}$ be an open covering of $M$  and let $\pi_\alpha: U_\alpha\rightarrow V_\alpha\subset \R^{2n}$ be submersions such that 
$d\pi_\alpha(\xi)$=0 and when $U_\alpha\cap U_\beta\neq \emptyset$
\[
\pi_\alpha\circ \pi_\beta^{-1}: \pi_\beta(U_\alpha\cap U_\beta)\rightarrow \pi_\alpha(U_\alpha\cap U_\beta)
\]
is diffeomorphic. On each $V_\alpha$, we can define an almost K\"ahler metric as follows. 

Let $(x^1, x^2, \cdots, x^{2n})$ be a local coordinate on $V_\alpha$. We can pull this back on $U_\alpha$ and still write them as $x^1, x^2, \cdots, x^{2n}$. Let $x^0$ be the coordinate along the leaves with $\xi=\frac{\p}{\p x^0}$. Then $(x^0, x^1, \cdots, x^{2n})$ forms a local coordinate on $U_\alpha$. Suppose we can write $\eta=dx^0-a_idx^i$ locally for some functions $a_i$, $1\leq i\leq 2n$. Since $\iota_\xi d\eta=0$, it is clear that $da_i(\xi)=0$ and $a_i$ is a function of $x^1, \cdots, x^{2n}$. We can then define
\[
\omega_\alpha:=\frac{1}{2}\left(\frac{\p a_i}{\p x^j}-\frac{\p a_j}{\p x_i}\right)dx^i\wedge dx^j= \omega_{ij}dx^i\wedge dx^j. 
\]
It is clear that $\omega_\alpha$ coincides with $\frac{1}{2}d\eta$ on  $\{x^0=const\}\subset U_\alpha$. We also get that  $\cD$ is spanned by the vectors of the form
\[
e_i=\frac{\p }{\p x^i}+a_i\frac{\p}{\p x^0}, i=1, 2, \cdots, 2n. 
\]
For any $p\in U_\alpha$ there is an isomorphism induced by $\pi_\alpha$
\[
d\pi_\alpha: \cD_p\rightarrow T_{\pi_{\alpha(p)}}V_\alpha; \; d\pi_\alpha (e_i)=\frac{\p}{\p x^i}.  
\]
An almost contact-complex structure $K$ can be written as
\[
\Phi e_i=\Phi_i^j e_j, \;\Phi_i^j\Phi_j^k=-\delta_{i}^k. 
\]
A simple computation shows that the condition $\Phi_\xi K=0$ implies that \[\frac{\p \Phi_i^j}{\p x^0}=0.\] Hence we can define an almost complex structure on $V_\alpha$ such that
\[
\Phi_\alpha \frac{\p}{\p x_i}=\Phi_i^j\frac{\p}{\p x_j}. 
\]
It is clear that $\Phi_\alpha$ is compatible with $\omega_\alpha$; thus we can define an almost  K\"ahler  metric $g^T_\alpha$ on $V_\alpha$. By this construction, 
\[
\pi_\alpha\circ \pi_\beta^{-1}: \pi_\beta(U_\alpha\cap U_\beta)\rightarrow \pi_\alpha(U_\alpha\cap U_\beta)
\]
gives an isometry of almost K\"ahler manifolds. The collection of almost K\"ahler metrics $\{V_\alpha, g^T_\alpha\}$ is called a transverse almost K\"ahler metric, which we denote by $g^T$ since they are isometric on the overlaps.  We also write $\nabla^T, Rm^T, Ric^T, R^T$ for its Levi-Civita connection, the curvature tensor, the Ricci tensor and the scalar curvature. 
It should be emphasized that,  when restricted on $\cD$,  $\{d\eta/2, \Phi, g\}$ on $U_\alpha$ is isometric to $\{\omega_\alpha, \Phi_\alpha, g^T_\alpha\}$ via $d\pi_\alpha: \cD_p\rightarrow T_{\pi_\alpha(p)}V_\alpha$. So  we can also define $\{\cD, d\eta/2, \Phi|_\cD, g|_\cD\}$ as an almost transverse K\"ahler structure via this isomorphism.  This isomorphism will play an important role in our computations in the following sections. 
The transverse scalar curvature  $R^T$ also lifts to $M$ as a global function.
When $(M, \eta, \xi, \Phi, g)$ is Sasaki,  it is easy to see that $\Phi_\alpha$ is a complex structure on $V_\alpha$, hence it defines  the {\it transverse complex structure} or {\it transverse holomorphic structure} on the Reeb foliation. In particular, we can choose coordinates charts $V_\alpha\subset \C^n$ with coordinates $(z^1, \cdots, z^n)$ such that
\[
\Phi_\alpha\frac{\p}{\p z_i}=\sqrt{-1} \frac{\p}{\p z_i};
\]
this with $\omega_\alpha$ gives its  transverse K\"ahler structure (for example see \cite{FOW} Section 3). The corresponding  charts $U_\alpha$ with coordinates $(x^0, z^1, \cdots, z^n)$
are called foliations charts. 

Now let $(M, \eta, \xi, \Phi)$ be a compact Sasaki manifold. 
\begin{defi} A $p$-form $\theta$ on $M$ is called basic if
\[
\iota_\xi \theta=0, L_\xi \theta=0.
\]
Let $\Lambda^p_B$ be the sheaf of germs of basic $p$-forms and $\Omega^p_B=\Gamma(S, \Lambda^p_B)$ the set of sections of $\Lambda^p_B$.  
\end{defi}

The exterior differential preserves basic forms. We set $d_B=d|_{\Omega^p_B}$. 
There is a natural splitting of $\Lambda^p_B\otimes \C$ such that
\[
\Lambda^p_B\otimes \C=\oplus \Lambda^{i, j}_B,
\]
where $\Lambda^{i, j}_B$ is the bundle of type $(i, j)$ basic forms. We thus have the well defined operators (for example, see Section 4 \cite{FOW})
\[
\begin{split}
\p_B: \Omega^{i, j}_B\rightarrow \Omega^{i+1, j}_B,\\
\bar\p_B: \Omega^{i, j}_B\rightarrow \Omega^{i, j+1}_B.
\end{split}
\]
Then we have $d_B=\p_B+\bar \p_B$. 
Set $d^c_B=\frac{1}{2}\i\left(\bar \p_B-\p_B\right).$ It is clear that
\[
d_Bd_B^c=\i\p_B\bar\p_B, d_B^2=(d_B^c)^2=0.
\]
Using the foliation chart, we can also check that
\[
d_B^c=\frac{1}{2}\i^{i-j}\Phi\circ d_B\; \mbox{on} \;\Omega^{i, j}_B. 
\]

 As in \cite{GZ1, GZ2}, we consider the space of Sasaki metrics as follows
\[
\cH=\{\phi\in C^\infty_B(M), \eta_\phi\wedge (d\eta_\phi)^n\neq 0, \eta_\phi=\eta+d^c_B\phi
\},\]
which can be viewed as an analogue of the space of K\"ahler metrics in a fixed K\"ahler class studied in \cite{Mabuchi87, Semmes96, Donaldson99}. 
For any $\phi\in \cH$, we can define a new Sasaki metric $(\eta_\phi, \xi, \Phi_\phi, g_\phi)$ with the same Reeb vector field $\xi$ such that
\[
\eta_\phi=\eta+d^c_B\phi, \Phi_\phi=\Phi-\xi\otimes d^c_B\phi \circ \Phi.
\]
Note that $\cD$ varies and the K-contact structure $\Phi_\phi$ varies when $\phi$ varies but $(\eta_\phi, \xi, \Phi_\phi, g_\phi)$ has the same transverse holomorphic structure as $(\eta, \xi, \Phi, g)$ (Prop. 4.1 in \cite{FOW}, \cite{BG}); for example, we have the relation when $\phi$ varies, 
\[
\Phi\circ d_B=\Phi_\phi\circ d_B. 
\]
On the contrary,  if $(\tilde \eta, \xi, \tilde \Phi, \tilde g)$ is another Sasaki structure  with the same Reeb vector field $\xi$ and the same transverse K\"ahler structure, then there exists a unique function $\phi\in \cH$ up to addition of a constant (e.g. \cite{EKA}) such that \[d\tilde \eta=d\eta_\phi=d\eta+\i \p \bar \p \phi.\]
We shall denote this set by $\cS(\xi, \bar J)$ as in \cite{BGS1}, where $\bar J$ denotes a fixed transverse homomorphic structure on $\nu(\cF_\xi)$. 
Note that $\tilde \eta$ does not have to be $\eta_\phi$, but the transverse geometry of $\tilde g$ is determined by $d\tilde \eta, \bar J$. Hence 
$\cS(\xi, \bar J)$ can be viewed as the analogue of the set of K\"ahler metrics in a fixed K\"ahler class.  Boyer-Galicki-Simanca \cite{BGS1, BGS2} proposed to seek the extremal Sasaki metrics  to represent $\cS(\xi, \bar J)$, by extending Calabi's extremal problem to Sasaki geometry (see \cite{FOW} also).  A metric is called {\it a transverse extremal K\"ahler metric}.  To seek an extremal metric in $\cS(\xi, \bar J)$
is then reduced to find a function $\phi\in \cH$ such that the Sasaki metric $(\eta_\phi, \xi, \Phi_\phi, g_\phi)$ is transverse extremal.  

We conclude this section with the following proposition, which will be used in the following sections.  Let $(M, g)$ be a Sasaki metric and $g^T$ be its transverse  K\"ahler metric. Let  $U_\alpha, (x^0, z^1, \cdots, z^n)$ be a foliation chart and $\pi_\alpha: U_\alpha\rightarrow V_\alpha$ such that
\[
d\pi_\alpha: D_p\rightarrow T_{\pi_\alpha(p)}V_\alpha. 
\]
Let 
\[g^T_{i\bar j}=g^T\left(\frac{\p}{\p z_i}, \frac{\p}{\p \bar z_{j}}\right)
\]
and $g^{i\bar j}_T$ be its inverse. 
It is clear that,  for a basic function $\phi$,
\[
\t^T \phi=g^{i\bar j}_T \frac{\p^2 \phi}{\p z_i\p \bar z_j} 
\]
lifts to $M$ hence defines a function on $M$. 
 Let $\phi, \psi$ be  basic functions, then $d\pi_\alpha(\nabla \phi)=\nabla^T\phi$ and we have
 \[
  (d_B \phi, d_B\psi)_g=(\nabla \phi, \nabla \psi)_g=(\nabla^T \phi, \nabla^T \psi)_{g^T}. \] It is also clear that for a basic function $\phi$,
 \[
 \t^T \phi=\t \phi.
 \]

\section{The Space of Sasaki Metrics as a Symmetric Space}

In this section we show that $\cH$  is an infinite-dimensional symmetric space, as in K\"ahler setting \cite{Mabuchi87, Semmes96, Donaldson99}.
It can be viewed as the negatively curved dual space of $\cG$, see  \cite{Donaldson99} for the K\"ahler setting.
 
Let $(M, \xi, \eta,  \Phi)$ be a compact Sasaki manifold. We shall briefly recall the geometric structure on $\cH$ introduced in \cite{GZ1}. 
For any $\phi\in \cH$, we can define a metric on $T_\phi\cH$
\[
\langle \psi_1, \psi_2\rangle_\phi=\int_M \psi_1\psi_2 d\mu_\phi, \forall \psi_1, \psi_2\in T\cH,
\]
where $d\mu_\phi=(2^nn!)^{-1}\eta_\phi\wedge (d\eta_\phi)^n$ is the volume form determined by $g_\phi$. If $\phi(t): [0, 1]\rightarrow \cH$ is a path, 
the geodesic equation can be written as 
\[
\ddot\phi-\frac{1}{4}|d_B\dot\phi|^2_\phi=0.
\]
Thus  $\psi(t)$ is a field of tangent vectors along the path $\phi(t)\in \cH$, the covariant derivative along the path is given by
\[
D_{\dot\phi}\psi=\frac{d}{dt}\psi-\frac{1}{4}(d_B \dot\phi, d_B\psi)_\phi.
\] 
Guan-Zhang  proved  \cite{GZ1} that this connection is torsion free and compatible with the metric
\[
\frac{d}{dt}\|\psi\|^2_\phi=2\langle D_t \psi, \psi \rangle;
\]
They also proved that  the corresponding  sectional curvature of $\cH$ is nonpositive and there is a Riemannian decomposition
\[
\cH=\cH_0\times \R,
\]where $\cH_0=\cH/\R$. 

\begin{theo}$\cH$ is an infinite dimensional symmetric space; the curvature of the connection $D$ is covariant constant. At a point $\phi\in \cH$ the curvature is given by
\[
R_\phi(\psi_1, \psi_2)\psi_3=-\frac{1}{16}\{\{\psi_1, \psi_2\}_\phi, \psi_3\}_\phi,
\]
where $\{   \;,   \;\}_\phi$ is the Possion bracket on $C^\infty_B(M)$ induced by the contact form $\eta_\phi$. 
\end{theo}

\begin{proof} First we observe that $(M, \eta_\phi)$ are all 
equivalent as contact structures.
Let $\phi(t)$ be a path starting at $0$ in $\cH$ and consider the $t$ dependent vector field 
\[
X_t=-\frac{1}{4} \nabla_{g(t)} \dot \phi,
\]
where $g(t)$ is the metric determined by $(M, \xi, \eta(t), \Phi(t))$. Let $f_t: M\rightarrow M$
be the 1-parameter family of diffeomorphisms obtained by integrating $X(t)$ with $f_0=1_M$.
Then we compute that 
\[
\frac{d}{dt} \left(f^{*}_t\eta(t)\right)=f_t^{*}\left(L_{X_t}\eta(t)+\frac{d}{dt}\eta(t)\right)
\]
For any $t$, $d\phi(t) (\xi)=0$, we get that $d\dot\phi (\xi)=0$. Hence $\langle \xi, \nabla_{g(t)} \dot \phi\rangle_{g(t)}=0$. This implies that $X_t\in \cD(t)=ker (\eta(t))$. So we can compute
\[
L_{X_t} \eta(t)=\iota_{X_t} d\eta(t)+d(\iota_{X_t}\eta(t))=\iota_{X_t}d\eta(t). 
\]
For each $t$ fixed and $Y\in  TM$,
\[
\begin{split}
\iota_{X_t}d\eta(t)(Y)&=d\eta(t)(X_t, Y)=-2\langle X_t, \Phi(t)Y\rangle_{g(t)}=\frac{1}{2}\langle \nabla \dot \phi, \Phi(t) Y\rangle\\
&=\frac{1}{2}d_B\dot\phi(\Phi(t)Y)=-\frac{1}{2}(\Phi(t)\circ d_B)\dot\phi (Y)=-d_B^c\dot \phi(Y). 
\end{split}
\]
On the other hand, the $t$-derivative of $\eta(t)$ is $d_B^c\dot\phi$. So we have
\[
\frac{d}{dt} f_t^{*}(\eta(t))=f_t^{*}\left(L_{X_t}\eta(t)+\frac{d}{dt}\eta(t)\right)
=0,
\]
i.e. the diffeomorphism $f_t$ gives the desired contactomorphism from $(M, \eta_0)$ to $(M, \eta(t))$.
Now let $\cY\subset \cH\times Diff(M)$ be the set of pairs $(\phi, f)$ such that $f^{*}\eta_\phi=\eta.$ This is a principal bundle over $\cH$ with structure group $\cG$. Then the discussion above shows that
the connection $D$ on the tangent space of $\cH$ is induced from a $\cG$ connection on $\cY\rightarrow \cH$ via the action of $\cG$ on $C^\infty_B(M)$; that is, we have a connection preserving bundle isomorphism
\[
T\cH=\cY\times_\cG C^\infty_B(M). 
\]
And the connection $D$ is compatible with the metric since the $L^2$ norm on $C^\infty_B(M)$ is $G$ invariant.
Now we compute the curvature tensor of $\cH$. To do this we consider a 2-parameter family $\phi(s, t)$ in $\cH$, and a vector field $\psi(s, t)$ along $\phi(s, t)$. We denote $s$ and $t$ derivatives by suffixes $\phi_s, \phi_t$ etc. The curvature is given by
\[
R_\phi(\phi_s, \phi_t)\psi=(D_sD_t-D_tD_s)\psi.
\]
It is  clear that $R(\phi_s, \phi_t)\psi$ is linear in $\phi, \psi$. 
It is also clear  that all functions involved are basic functions. Hence we can do computation by using the transverse K\"ahler structure defined by $\{\cD_\phi, \eta_\phi, \Phi_\phi, g_\phi\}$. 
We can write, for example, in a foliation chart  $U_\alpha\subset M$, 
\[
( d_B \psi_1, d_B\psi_2)_\phi=( \nabla^T\psi_1, \nabla^T\psi_2)_{g^T},
\]
where $\nabla^T$ is the Levi-Civita connection of the transverse K\"ahler metric ${g^T}$ defined by $\{\cD_\phi, \eta_\phi, \Phi_\phi, g_\phi\}$. 
We also consider the projection $\pi_\alpha: U_\alpha \rightarrow V_\alpha$ and the corresponding complex structure on $V_\alpha$ is denoted as $\Phi^\alpha_\phi$. 

Expanding out, we compute
\begin{equation}\label{E-3-a}
\begin{split}
4 R_\phi(\phi_s, \phi_t)\psi=&\left(\frac{\p}{\p t}( \nabla^T \psi, \nabla^T \phi_s )_{g^T}-\frac{\p}{\p s}( \nabla^T \psi, \nabla^T \phi_t)_{g^T}\right)\\&+\left((\nabla^T\psi_s, \nabla^T\phi_t)_{g^T}-( \nabla^T\psi_t, \nabla^T\phi_s)_{g^T}\right)\\
&+\left( \nabla^T( \nabla^T \psi, \nabla^T \phi_s)_{g^T}, \nabla^T\phi_t\right)_{g^T}-\left( \nabla^T( \nabla^T \psi, \nabla^T \phi_t)_{g^T}, \nabla^T\phi_s\right)_{g^T}
\end{split}
\end{equation}
On the other hand, we can write the Possion bracket $\{\; , \;\}$ as follows. For any $f, h\in C^\infty_B(M)$, 
\[
\{f, h\}_\phi=\eta_\phi([X_f, X_h])=-d\eta_\phi(X_f, X_h)=2( \nabla^T f, \Phi_\phi^\alpha\nabla^{T} h)_{g^T}. 
\]
Hence we compute 
\begin{equation}\label{E-3-b}
\frac{1}{4}\{\{\phi_s, \phi_t\}, \psi\}_\phi=( \nabla^T( \nabla^T\phi_s, \Phi^{\alpha}_\phi \nabla^T \phi_t)_{g^T}, \Phi^{\alpha}_\phi\nabla^T\psi)_{g^T}. 
\end{equation}
Now we claim that the right hand side of \eqref{E-3-a} coincides with the right hand side of \eqref{E-3-b}.
Note that all the quantities above are only involved with the transverse K\"ahler structure and we can express all these quantities in terms of the Kahler metric $g^T_\alpha$ on $V_\alpha$. 
Hence this reduces the computations to the Kahler setting and the same argument  in \cite{Donaldson99} (Page 18) can be carried over here directly without any change. The only difference is that the K\"ahler form of $g^T_\alpha$ is $d\eta_\phi/2$ while the Possion bracket is defined by $\eta_\phi$ and $d\eta_\phi$. So there is difference of a constant factor $4$ and we get
\[
R_\phi(\phi_s, \phi_t)\psi=-\frac{1}{16}\{\{\phi_s, \phi_t\}_\phi, \psi\}_\phi. 
\]
The expression of the curvature tensor in terms of Possion brackets shows that $R_\phi$ is invariant under the action of the group $\cG$; it follows that $R_\phi$ is covariant constant and hence $\cH$ is indeed an infinite-dimensional symmetric space. 
\end{proof}

\begin{rmk}
It is clear that the Riemannian  decomposition $\cH=\cH_0\times \R$ corresponds to the Lie algebra decomposition $C^\infty_B(M)=C^\infty_{B, 0}(M)\oplus \R$. 
When $(M, \xi, \eta)$ is quasiregular, then $\xi$ generates one parameter subgroup of $\cG$  which lies in the center of $\cG$ (in quasiregular case, the orbits of $\xi$ are compact, hence are all circles, then $\xi$ generates a circle action on $M$ which preserves $\eta$), then we have the group isomorphism 
\[
\cG\approx \cG/S^1 \times S^1, 
\]
which corresponds the Lie algebra decomposition  $C^\infty_B(M)=C^\infty_{B, 0}(M)\oplus \R$. 
\end{rmk}

\section{Transverse Scalar Curvature as a Moment Map}

In this section we show that the transverse scalar curvature of a Sasaki metric is a moment map with respect to the strict contactomorphism group $\cG$. 
When $(V, \omega)$ is a symplectic manifold and let $\cJ$ be the space of almost complex structures which are compatible with $\omega$. Then   scalar curvature $R: \cJ\rightarrow C^\infty(M)$ is a moment map of symplectomorphism group  of $(V, \omega)$ which acts on  $\cJ$  \cite{Fujiki90, Donaldson99}. This point of view of moment map can be carried over to Sasaki geometry with only slight modification. 

First let us recall the definition of a moment map. 
Let $(V, \omega)$ be a symplectic manifold with a symplectic form $\omega$. Suppose that a Lie group $G$ acts on $V$ via symplectomorphisms. 
Let $\g$ be the Lie algebra of $G$. Then  any $\zeta\in \g$ induces a one-parameter subgroup $\{\phi(t)\}$ in $G$, and $\{\phi(t)\}$ induces a vector field $X_\zeta$ on $V$ since $G$ acts on $V$. 
A moment map 
for the $G$-action on $(V, \omega)$ is a map
$\mu: V\rightarrow \g^{*}$ such that
\[
d\langle \mu, \zeta\rangle =\iota_{X_\zeta}\omega.
\]
Here $\langle \mu, \zeta\rangle$ is the function from $V$ to $\R$ defined by $\langle \mu, \zeta\rangle(x)=\langle \mu(x), \zeta\rangle$. We also require that the moment map $\mu$ is $G$-equivariant with respect to the co-adjoint action of $G$ on $\g^{*}$.

Consider a compact contact manifold  $(M^{2n+1}, \eta)$. 
Let $\cK$ be the space of  $K$-contact-complex structures on $M$ which are compatible with $\eta$;  we shall assume $\cK$ is not empty. 
We shall show that the space $\cK$ can be endowed with the structure of an infinite-dimensional K\"ahler manifold. Let $\cG$ be the strict contactomorphism group which preserves the contact form $\eta$. Then $\cG$ acts on $\cK$ via
\[
(f, \Phi)\rightarrow f^{-1}_{*}\Phi f_{*}.
\]
For simplicity we consider $\Phi\in \cK$ is integrable. Let $\cK_{int} \subset \cK$ defined as
\[
\cK_{int}=\{\Phi\in \cK: \Phi\; \mbox{ is integrable}\}. 
\]
We assume $\cK_{int}$ is not empty. 
We want to identify a moment map for the action of $\cG$ on $\cK_{int}$. For each $\Phi\in \cK_{int}$, $(M, \eta, \Phi)$ defines a Sasaki metric $g$, and let $\{\cD, \frac{1}{2}d\eta, \Phi|_\cD\}$ be its transverse K\"ahler metric.  We have 
\begin{theo}\label{T-2}The map $\Phi\rightarrow R^{T}(\Phi)$ is an {\it equivariant} moment map for the $\cG$-action on $\cK_{int}$, under the natural pairing:
\[
(R^T, H)\rightarrow \int_M R^T H d\mu_\eta. 
\]
\end{theo}
First observe that
\[
T_K\cK=\{A: TM\rightarrow TM; A\xi=0, L_\xi A=0, A\Phi+\Phi A=0, d\eta(\Phi A \cdot, \;\cdot)+d\eta( \cdot, \; \Phi A \cdot)=0\}.
\]
Note that if $A\in T_K\cK$, then $\Phi A\in T_K\cK$; hence we can define a natural almost structure  $J_\cK: T_K\cK\rightarrow T_K\cK$ on $\cK$ such that
\[
J_\cK A=\Phi A. 
\] 
We can define a natural metric on $\cK$.  For any $A\in T_K\cK$,  we can identify $A$ with
\[
A(X, Y)=g_\Phi(AX, Y), X, Y\in TM. 
\]
It is clear that $A$ is anti-$\Phi$ invariant and symmetric. We can  identify $A$ with a section of  $T^{*}M\otimes T^{*}M$.  The metric $g_\Phi$ induces a metric on  $T^{*}M\otimes T^{*}M$.  So we can define a metric on $T_\Phi\cK$
 by integration over $M$,
\[
g_\cK(A, B)=\int_M \langle A, B\rangle_{g_{ \Phi}} d\mu.
\]
For $p\in M$,  if $\cD_p=span\{e_i: 1\leq i\leq 2n\}$, and $A_p:\cD_p\rightarrow \cD_p$ such that $A_pe_i=A_i^je_j$. Then $\langle A, B \rangle_p=trace(A_pB_p)=A_i^jB_j^i$.
It is clear that $g_\cK$ is compatible with $J_\cK$; hence it defines an Hermitian metric on $\cK$.  As in the (almost) K\"ahler setting, $(\cK, g_\cK, J_\cK)$ induces a natural K\"ahler structure on $\cK$. 
\begin{prop}$(\cK, g_\cK, J_\cK)$ is K\"ahler and $\cK_{int}$ is an analytic submanifold.
\end{prop}

Note that $\Phi\in \cK$ can be identified as an almost complex structure on $\cD$ which is invariant under the action generated by $\xi$,  hence $\cK$ can be viewed the space of sections of a fibre bundle on $M$ with fibre the complex homogeneous space $Sp(2n, \R)/U(n)$. Then $\cK$ inherits a complex structure from that of the fibre. Indeed, there is a global holomorphic coordinate chart if we use the ball model of the Siegel upper half space in the usual way, see \cite{Fujiki90} Theorem 4.2 and \cite{Donaldson99} for more details.

Now we are in the position to prove Theorem \ref{T-2}.
\begin{proof}
 Fix a K-contact-complex structure $\Phi \in \cK_{int}$; let
\[
\cP: C^\infty_B(M)\rightarrow \Gamma(T\cK) 
\]
be the operator representing the infinitesimal action of $\cG$ on $\cK_{int}$ and
\[
\cQ: \Gamma(T\cK)\rightarrow C^\infty_B(M)
\]
be the operator which represents the derivative of $\Phi\rightarrow R^T(\Phi)$. By the definition of a moment map we need to show that for all $\phi\in C^\infty_B(M)$, $A\in \Gamma(T\cK)$, 
\[
g_\cK(J_\cK  \cP(\phi), A)=\langle \phi, \cQ(A)\rangle.
\]
The operator $\cP$ can be factored as $\cP=\cP_2\cP_1$, where $\cP_1$ maps $\phi$ to the Hamiltonian vector field $X_\phi$ and $P_2$ maps a vector filed $X$ to the infinitesimal variation given by the Lie derivative $L_X \Phi$.  
We choose a coordinate chart $U_\alpha=\{(x^0, x^1, \cdots, x^{2n})\}$ as in Section 2, such that
\[
\eta=dx^0-a_idx^i, \xi=\frac{\p}{\p x_0}, \cD=span\left\{\frac{\p}{\p x_i}+a_i\frac{\p}{\p x^0}, 1\leq i\leq 2n\right\}. 
\] 
Let $V_\alpha=\{(x^1, \cdots, x^{2n})\}$ such that $\pi_\alpha: U_\alpha\rightarrow V_\alpha$ is a submersion. 
Recall  $\omega_\alpha=d\eta/2=(\omega_{ij}), \Phi_\alpha=(\Phi_{i}^j)$  and $g_\alpha^T$  the transverse K\"aher structure. For simplicity we suppress $\alpha$ and write
$g^T_\alpha=(g_{ij})=(\omega_{ik}\Phi^k_j)$, $1\leq i, j\leq 2n$.  It is clear that on $U_\alpha$, we can write $A=(A^j_i)$ for some functions $A^j_i$ independent of $x^0$, as we have discussed in Section 2. Hence we can define $A_\alpha$ on $V_\alpha$, which we still denote as $A_\alpha=(A^j_i)$. 
Note that  all quantities involved in the computations can be written locally on $V_\alpha$ and we have the isometry of  $\{\cD, \frac{1}{2}d\eta, \Phi|_\cD\}$ with $g^T$. So we can do all (local) computations on $V_\alpha$.
Since $\iota_{X_\phi}d\eta=-d\phi$,  we have
\[
\omega\left(X_\phi, \frac{\p}{\p x_j}\right)=-\frac{1}{2}\frac{\p \phi}{\p x_j}.
\] 
It follows that
\[
X_\phi=X^i\frac{\p }{\p x_i}=-\frac{1}{2}\Phi^j_kg^{ki}\frac{\p \phi}{\p x_j}. 
\]
Suppose that
\[
L_{X_\phi}\Phi=B^i_j\frac{\p}{\p x_i}\otimes dx_j;
\]
then we have
\[
B^i_j\frac{\p }{\p x_i}=L_{X_\phi}\Phi\left(\frac{\p}{\p x_j}\right)= L_{X_\phi}\left(\Phi\frac{\p}{\p x_j}\right)-\Phi\left(L_{X_\phi} \frac{\p}{\p x_j}\right).
\]
Hence
\[
B^i_j=\frac{1}{2} \Phi_j^k\Phi^p_lg^{li}\phi_{kp}-\frac{1}{2}\Phi^i_kg^{kl}\Phi^p_l\phi_{pj}.
\]
It follows that
\[
\cP(\phi)=\frac{1}{2}\left(\Phi_i^k\Phi_j^l\phi_{kl}-\phi_{ij}\right).
\]
We can then compute the pairing 
\[
g_\cK(J_\cK\cP(\phi),  A)=-\int_M \langle \cP(\phi), J_\cK A\rangle d\mu=\int_M g^{jk}g^{pl}\Phi^i_p\phi_{ij}A_{kl}d\mu,
\]
where $A_{ij}=g_{ik}A^k_j$, $1\leq i, j, k, l, p\leq 2n$. Now we  compute $\cQ(A)$. Since $\left(\delta g^T\right)_{ij}=\omega_{ik}A^k_j=\Phi^k_iA_{kj}$, the well-known formula of the variation of the scalar curvature $R^T$ is given by
\[
\delta R^T=-2g^{ij}g^{kl}\left(\delta g\right)_{il}R^{T}_{jk}+g^{ik}g^{jl}\left(\delta g^T\right)_{ij, kl}-g^{ij}g^{kl}\left(\delta g^T\right)_{ij, kl}.
\]
Note that the metric $g^T$ and the Ricci curvature  $R^T_{ij}$ are  $\Phi$-invariant and $A$ is anti $\Phi$-invariant, we can get that 
\[
g^{ij}g^{kl}\left(\delta g^T\right)_{il}R^{T}_{jk}=0, g^{ij}\left(\delta g^T\right)_{ij}=0. 
\]
It follows that
\[
\cQ(A)=\delta R^T=g^{ik}g^{jl}\left(\delta g^T\right)_{ij, kl}=g^{ik}g^{jl} \left(\Phi_i^pA_{pj}\right)_{, kl}.
\]
We can compute, integration by parts, 
\[
\langle \phi, \cQ(A)\rangle=\int_M \phi g^{ik}g^{jl}\left(\Phi_i^pA_{pj}\right)_{, kl}d\mu=\int_M g^{ik}g^{jl} \Phi_i^p \phi_{kl}A_{pj}d\mu.
\]
The integration by parts above can be justified as in Proposition 4.5 \cite{He}, by noting that all the integrands involved are basic, namely, are invariant under the group action generated by $\xi$. 
Hence we get the desired equality
\[
g_\cK(J_\cK \cP(\phi),  A)=\langle \phi, \cQ(A)\rangle.
\]
\end{proof}

\begin{rmk}The similar results should hold for  $\cK$, the space of $K$-contact complex structures which are compatible with a fixed contact form. One should consider the transverse Hermitian scalar curvature defined by the almost transverse K\"ahler structure,  but   have to take the Nijenhuis tensor into account, as in \cite{Donaldson99} for  the almost K\"ahler setting. We shall assume integrability for simplicity. 
\end{rmk}

\section{Canonical Metrics in Sasaki Geometry}
Calabi's extremal metric problem \cite{Calabi82, Calabi85} in K\"ahler geometry are closely related to the geometry of the space of K\"ahler metrics and the space of complex structures which are compatible with a fixed K\"ahler form. We refer the reader to the papers of S.K. Donaldson \cite{Donaldson97, Donaldson99} for details. 
Calabi's extremal  problem can also be extended to  Sasaki geometry, see \cite{BGS1, BGS2} for example. 

We have seen that the standard picture in \cite{Donaldson97, Donaldson99} can be applied to Sasaki geometry; hence
canonical metric problems in Sasaki geometry are also closely related to the geometry of $\cH$ and $\cK$.  We shall roughly repeat the  picture described in \cite{Donaldson97, Donaldson99} for Sasaki context as follows. 

If $(V, \omega, J)$ is a K\"ahler manifold and assume there is an action of a compact connected group $G$ on $V$ which preserves the K\"ahler structure. Let $\mu$ be the corresponding moment map. This induces a holomorphic action of the complexified group $G^{\C}$. Then the Kempf-Ness theorem relates the complex quotient by $G^{\C}$ to the symplectic reduction by $G$ (\cite{DK}).

Now let $(M, \eta, \xi)$ be a Sasaki manifold.
The strict contactomorphism group $\cG$ acts on the K\"ahler manifold $\cK^{int}$ as holomorphic isometries, and the transverse scalar curvature $R^T: \cK^{int}\rightarrow C^\infty_{B}(M)$ is an equivariant moment map of $\cG$-action on $\cK^{int}$.   
Let  $m=R^T-\underline{R}$, where $\underline{R}$ is the average of the transverse scalar curvature, depending only on the basic  class $[d\eta]_B$.   Then $m:\cK^{int}\rightarrow C^\infty_B(M)$ is also a moment map of $\cG$. 
If there is a {\it complexification} group $\cG^c$ of $\cG$, then by the {\it standard picture},   a constant transverse scalar curvature metric, which is a zero point of the moment map $m$, corresponds to a stable complex orbit of $\cG^c$ action and we expect the identification
\[
K^{int}_s/\cG^c=m^{-1}(0)/\cG. 
\]

In general,  $\cG^c$ might not exist. However we can complexify the Lie algebra of $\cG$ and it acts on $\cK^{int}$ automatically since $\cK^{int}$ is a complex manifold. At each point $\Phi\in \cK^{int}$ we get a subspace of $T_\Phi\cK^{int}$ spanned by this complexification action and these subspaces form an integrable, holomorphic distribution on $\cK^{int}$. Thus we get a distribution of $\cK^{int}$ which plays the role of the complex orbits. By definition, the infinitesimal action of $\phi$ on $\cK^{int}$ is given by $L_{X_\phi} \Phi$ for any $\Phi\in T_\Phi\cK^{int}$, then 
the infinitesimal action of a function $\i \phi\in C^\infty_B(M)\otimes \C$ is $L_{\Phi X_\phi}\Phi$,  the natural action of $\Phi X_\phi$ on $\Phi$. Thus the geometric effect of applying $\i \phi$ is the same as keeping the  transverse holomorphic structure induced by $\Phi$ fixed and varying the transverse symplectic form $\frac{1}{2}d\eta$ to
\[
\frac{1}{2}d\tilde\eta=\frac{1}{2}d\eta-L_{\Phi X_\phi} \left(\frac{1}{2}d\eta\right)=\frac{1}{2}d\eta+\i \p \bar \p \phi,
\]
which corresponds to the space of Sasaki metrics $\cH$. 

We shall recall the definition of K-energy in Sasaki geometry \cite{Mabuchi86, GZ2}, which is defined on $\cH$ by specifying its variation
\[
\delta \cM=-\int_M \delta \phi (R^T-\underline{R})d\mu_\phi.
\]
A critical point of $\cM$ is a constant transverse scalar curvature and $\cM$ is convex along geodesics in $\cH$ \cite{GZ2}.

We can also ask the similar questions as in \cite{Donaldson99} regarding the existence of constant transverse scalar curvature metrics, which can be viewed as an analogue of the Hilbert criterion for stability in geometric invariant theory.

\begin{pr}The following are equivalent:

(1). There is no critical K\"ahler metric in $\cH_0$.

(2). There is an infinite geodesic ray $\phi_t\in \cH_0$, such that $t\rightarrow \infty$, the derivative of K-energy is less than zero along the geodesic ray $\phi_t$,
\[
\int_M \dot\phi (\underline{R}-R^T)d\mu_\phi<0.
\]

(3). For any point $\phi\in \cH_0$ there is a geodesic ray as in (2) starting at $\phi_0$. 
\end{pr}

\bibliographystyle{amsplain}

\begin{thebibliography}{10}

\bibitem{BG} C.P. Boyer,  K. Galicki; {\it Sasaki geometry}, Oxford Mathematical Monographs. Oxford University Press, Oxford, 2008. xii+613 pp. 


\bibitem{BGS1} C.P. Boyer,  K. Galicki, S.R. Simanca; {\it Canonical Sasaki metrics}, Comm. Math. Phys. 279 (2008), no. 3, 705--733.

\bibitem{BGS2}C.P. Boyer,  K. Galicki, S.R. Simanca;  {The Sasaki cone and extremal Sasaki metrics},  Riemannian topology and geometric structures on manifolds, 263--290, Progr. Math., 271, BirkhŠuser Boston, Boston, MA, 2009.

\bibitem{Calabi82}E. Calabi; {\it Extremal K\"ahler metric}, in {\it Seminar of Differential Geometry},
ed. S. T. Yau, Annals of Mathematics Studies {\bf 102}, Princeton
University Press (1982), 259-290.

\bibitem{Calabi85} E. Cakabi; {\it Extremal K\"ahler metrics II,}  Differential geometry and complex analysis, 95--114, Springer, Berlin, 1985.

\bibitem{CalabiChen} E. Calabi, X. Chen;  {\it The space of K\"ahler
metrics II}, J. Differential Geom.  61  (2002),  no. 2, 173--193.

\bibitem{Chen00} X. Chen; {\it The space of K\"ahler metrics}, J. Differential Geom. 56 (2000), no. 2, 189--234.

\bibitem{Donaldson97}S.K. Donaldson; {\it Remarks on gauge theory, complex geometry and $4$-manifold topology,} Fields Medallists' lectures, 384--403, World Sci. Ser. 20th Century Math., 5, World Sci. Publ., River Edge, NJ, 1997.

\bibitem{Donaldson99}S.K.  Donaldson; {\it Symmetric spaces, K\"ahler geometry and Hamiltonian dynamics.} Northern California Symplectic Geometry Seminar, 13--33, Amer. Math. Soc. Transl. Ser. 2, 196, Amer. Math. Soc., Providence, RI, 1999.

\bibitem{Donaldson01}S.K. Donaldson; {\it Scalar curvature and projective embeddings. I,} J. Differential Geom. 59 (2001), no. 3, 479--522. 

\bibitem{Donaldson02}S.K. Donaldson; {\it Scalar curvature and stability of toric varieties,} J. Differential Geom. 62 (2002), no. 2, 289--349.

\bibitem{DK}  S. K. Donaldson, P. B. Kronheimer; {\it The geometry of four-manifolds},  Oxford Mathematical Monographs. Oxford Science Publications. The Clarendon Press, Oxford University Press, New York, 1990.

\bibitem{EKA}  A. El Kacimi-Alaoui; {\it Operateurs transversalement elliptiques sur un feuilletage riemannien
et applications,} Compositio Math. 79, (1990), 57-106.

\bibitem{Fujiki90}A. Fujiki; {\it The moduli spaces and K\"ahler metrics of polarised algebraic varieties}, Suguku 42 (1990), 231-243; English transl., Sugaku Expositions 5 (1992),  173-191.

\bibitem{Futaki83}A. Futaki, {\it  An obstruction to the existence of Einstein K\"ahler metrics,} Invent. Math. 73 (1983), no. 3, 437--443. 

\bibitem{FOW}A. Futaki, H.  Ono, G. Wang; {\it Transverse K\"ahler geometry of Sasaki manifolds and toric Sasaki-Einstein manifolds},  arXiv:math/0607586. 

\bibitem{GMSY} J. P. Gauntlett, D. Martelli, J. Sparks, S.T. Yau; 
 {\it Obstructions to the Existence of Sasaki-Einstein Metrics}, Comm. Math. Phy. 273 (2007), 803-827.

\bibitem{GZ1}P. Guan, X. Zhang; {\it  A geodesic equation in the space of Sasaki metrics,} to appear in Yau's Preceedings.

\bibitem{GZ2} P. Guan, X. Zhang; {\it Regularity of the geodesic equation in the space of Sasaki metrics},  arXiv:0906.5591.

\bibitem{He}W. He, {\it The Sasaki-Ricci flow and compact Sasaki manifolds with positive transverse bisectional curvature}, arXiv:1103.5807.

\bibitem{L}A. Lichnerowicz, {\it Sur les transformations analytiques des vari\'et\'es k\'ahl\'eriennes compactes,} (French) C. R. Acad. Sci. Paris 244 (1957), 3011--3013.

\bibitem{Mabuchi86}  T. Mabuchi; {\it $K$-energy maps integrating Futaki invariants},  Tohoku Math. J. (2) 38 (1986), no. 4, 575--593. 

\bibitem{Mabuchi87} T. Mabuchi; {\it Some symplectic geometry on compact K\"ahler manifolds. I},  Osaka J. Math. 24 (1987), no. 2, 227--252. 

\bibitem{Mabuchi04}T. Mabuchi; {\it Stability of extremal K\"ahler manifolds,}  Osaka J. Math. 41 (2004), no. 3, 563--582.

\bibitem{MSY} D. Martelli, J. Sparks, S.T. Yau; {\it Sasaki-Einstein Manifolds and Volume Minimisation}, Commun.Math.Phys. 280 (2008), 611-673.

\bibitem{M} Y. Matsushima, {\it Sur la structure du groupe d'hom\'eomorphismes analytiques d\`une certaine vari\'et\'e kaehl\'erienne,} Nagoya Math. J. 11 (1957), 145-150.

\bibitem{MS} D. McDuff, D. Salamon; {\it Introduction to symplectic topology},  Oxford Mathematical Monographs. Oxford Science Publications. The Clarendon Press, Oxford University Press, New York, 1995. viii+425 pp.

\bibitem{RT} J. Ross, R. Thomas; {\it Weighted projective embeddings, stability of orbifolds and constant scalar curvature K\"ahler metrics},  arXiv:0907.5214. 

\bibitem{Semmes96}S. Semmes;  {\it Complex Monge-Amp\`ere and symplectic manifolds,} Amer. J. Math.  114  (1992),  no. 3, 495--550.

\bibitem{Sparks}J. Sparks, {\it Sasakian-Einstein manifolds}, arXiv:1004.2461.

\bibitem{Sze} G. Sz\'ekelyhidi;   {\it Extremal metrics and $K$-stability.} Bull. Lond. Math. Soc. 39 (2007), no. 1, 76--84.

\bibitem{Tian97} G. Tian;  {\it K\"ahler-Einstein metrics with positive scalar curvature},
Invent. Math.  130  (1997),  no. 1, 1--37.

\bibitem{Yau90} S.T. Yau; {\it Open problems in geometry.} Differential geometry: partial differential equations on manifolds (Los Angeles, CA, 1990), 1--28, Proc. Sympos. Pure Math., 54, Part 1, Amer. Math. Soc., Providence, RI, 1993.

\end{thebibliography}

\end{document}